\documentclass[leqno]{amsart}
\usepackage{srctex,color,amssymb,amsthm,amsmath,eucal,mathrsfs}

\renewcommand{\Re}{\operatorname{Re}}

\newcommand{\dd}{\mathrm{d}}
\newcommand{\loc}{\mathrm{loc}}

\newcommand{\R}{\mathbb R}

\newcommand{\ep}{\varepsilon}

\newcommand{\norm}[1]{\|  #1 \|}

\newcommand{\inner}[3][]{(#2,#3 )_{#1}}

\def\cL{\mathcal{L}}

\newtheorem{theorem}{Theorem}[section]
\newtheorem{lemma}[theorem]{Lemma}
\newtheorem{proposition}[theorem]{Proposition}

\newtheorem{corollary}[theorem]{Corollary}
\newtheorem{ass}[theorem]{Assumption}

\theoremstyle{definition}
\newtheorem{definition}[theorem]{Definition}

\theoremstyle{remark}
\newtheorem{remark}[theorem]{Remark}

\numberwithin{equation}{section}

\begin{document}

\title[Semilinear stochastic Volterra equations]{Existence, uniqueness and regularity for a class of semilinear stochastic Volterra equations with multiplicative noise}

\author{Boris Baeumer}
\address{Department of Mathematics and Statistics, University of Otago, PO Box 56, Dunedin, 9054, New Zealand.}
\email{bbaeumer@maths.otago.ac.nz}

\author{Matthias Geissert}
\address{TU Darmstadt, Angewandte Analysis, Schlossgartenstr. 7, 64289 Darmstadt, Germany.}
\email{geissert@mathematik.tu-darmstadt.de}

\author{Mih\'aly Kov\'acs}
\address{Department of Mathematics and Statistics, University of Otago, PO Box 56, Dunedin, 9054, New Zealand.}
\email{mkovacs@maths.otago.ac.nz}

\date{}


\subjclass[2010]{35B65, 60H15, 35R60, 45D05, 34A08}
\keywords{Volterra equation, stochastic partial differential equation, fractional differential equation, Wiener process, Gaussian noise, multiplicative noise, regularity, existence, uniqueness.}

\begin{abstract}
We consider a class of semilinear Volterra type stochastic evolution equation driven by multiplicative Gaussian noise. The memory kernel, not necessarily analytic, is such that the deterministic linear equation exhibits a parabolic character. Under appropriate Lipschitz-type and linear growth assumptions on the nonlinear terms we show that the unique mild solution is mean-$p$ H\"older continuous with values in an appropriate Sobolev space depending on the kernel and the data. In particular, we obtain pathwise space-time (Sobolev-H\"older) regularity of the solution together with a maximal type bound on the spatial Sobolev norm. As one of the main technical tools we establish a smoothing property of the derivative of the deterministic evolution operator family.
\end{abstract}
\maketitle

\section{Introduction}
We consider a stochastic evolution equation of Volterra type driven by multiplicative Gaussian noise given in the It\^o form
\begin{equation}\label{eq:problem}
\begin{aligned}
\dd u+\left(A\int\limits_0^t b(t-s)u(s)\dd s\right)\dd t&=F(u)\dd t+G(u)\dd
	W,\quad t>0,\\
	u(0)&=u_0.
\end{aligned}
\end{equation}
The process $\{u(t)\}_{t\in [0,T]}$, defined on a filtered probability space $(\Omega,\mathcal{F},\mathbb{P},\{\mathcal{F}_t\}_{t\ge 0})$ with a normal filtration $\{\mathcal{F}_t\}_{t\ge 0}$, takes values in a separable Hilbert space $H$ with inner product $(\cdot\, ,\cdot)$ and induced norm $\|\cdot\|$. The process $W$ is a nuclear $Q$-Wiener process with respect to the filtration
with values in some separable Hilbert space $U$. The operator $A:\mathcal{D}(A)\subset H\to H$ is assumed to be linear, unbounded, self-adjoint and positive definite. The main example we have in mind for $A$ is the Dirichlet Laplacian on $H=L^2(\mathcal{O})$, where $\mathcal{O}\subset \mathbb{R}^d$ is a spatial domain with smooth boundary. Throughout the paper the kernel $b$ is kept as general as possible but so that the deterministic, linear, homogeneous version of \eqref{eq:problem} exhibits a parabolic character. In particular, we assume that the Laplace transform $\hat{b}$ of $b$ maps the right half-plane into a sector around the real axis with central angle less than $\pi$, and $\hat{b}$ satisfies some regularity and growth conditions, see Assumption \ref{ass:convKernel} and Remark \ref{rem:realb} for the precise conditions on $b$. One of the important kernels that satisfies this assumption is the tempered Riesz kernel $b(t)=\frac{1}{\Gamma(\rho-1)}t^{\rho-2}e^{-\eta t}$, where $1<\rho<2$ and $\eta\ge 0$.
The main goal of the paper is to extend the results of \cite{JR} and \cite{KrLa}, where SPDEs without a memory term are considered, to the solution of \eqref{eq:problem} under analogous, appropriate Lipschitz and linear growths  assumptions on $f$ and $G$, see Assumption \ref{ass:main}. That is, in Theorem \ref{thm:eu}, we prove existence, uniqueness, and mean-$p$ H\"older regularity in time with values in fractional order spaces $\dot{H}^{\beta}$ associated with the fractional powers of $A$ (see, Subsection \ref{sc:ass}). Then corresponding pathwise regularity results immediately follow, see  Corollary \ref{cor:path}.

There are several approaches to stochastic partial differential equations and therefore to Volterra type stochastic partial differential equations as well. Firstly, one chooses a framework for infinite dimensional stochastic integration. One possibility, as in the present paper, is to choose abstract stochastic integration theory in Hilbert spaces, such as in \cite{DPZ} and \cite{PR}. Then, one has the option to consider a semigroup framework from \cite{DPZ} with a suitable state space that incorporates the history of the process as, for example, in \cite{Barbu,BoDP}. In the latter papers existence and uniqueness is established for a class of semilinear Volterra type SPDEs with multiplicative noise under partly more general and partly more restrictive assumptions then in this paper and space-time regularity is not investigated.

The other option for defining solutions to \eqref{eq:problem}, which we also choose to follow, is the resolvent family approach of Pr\"uss \cite{pruss} based on the Laplace transform. This has been mainly used to study linear equations with additive noise, see \cite{CDaPP,Karczewskabook,Karczewska2012,Karczewska2009,Sperlich}, with the exception of \cite{BoFa} and \cite{Keck} where a semilinear equation with additive noise and, respectively, multiplicative noise is considered. All these papers are mainly concerned with existence and uniqueness, and not so much with regularity, apart for some limited analysis in the linear additive case \cite{CDaPP,Sperlich}.

 Finally, an other possibility is to use Krylov's approach for stochastic integration in case $H$ is specifically $L^2(\mathcal{O})$ (more generally $L^p(\mathcal{O})$) where stochastic integrals are taken pointwise. This approach is taken, for example, in \cite{DLO} where a semilinear Volterra type equation is considered with the specific kernel $b(t)=\frac{1}{\Gamma(\rho-1)}t^{\rho-2}$. There the authors obtain regularity results (but not pathwise ones) which are the same flavour as ours as they show a balance between spatial and temporal regularity. However the results there are rather difficult to compare with ours because the different framework as it is also pointed out there.

The paper is organized as follows. In Section \ref{sec:prelim} we introduce notation, collect some necessary background material and state the main assumptions on the data in \eqref{eq:problem}. In Section \ref{sec:mr} we first introduce the mild solution concept. Theorem \ref{thm:eu} contains the main result of the paper by establishing existence, uniqueness and space-time regularity of the mild solution of \eqref{eq:problem}. The proof uses a fixed point argument in a space of low regularity together with a regularity bootstrapping to obtain the highest possible regularity. The key result for the bootstrapping argument is stated in Proposition \ref{prop:regularity}. The section further contains a pathwise space-time regularity result (Corollary \ref{cor:path}) and ends with Subsection \ref{ss:an} where the special case of additive noise is discussed. The latter is in fact important as maximal type space-regularity for the linear equation is often assumed when handling the nonlinear equation, see, for example \cite{Barbu,BoFa}.  Hence, in Corollary \ref{cor:tcn}, we provide a class of convolution kernels $b$ and give conditions on $Q$ so that it holds. In Subsection \ref{ss:an} we also demonstrate that the abstract framework of the paper is flexible enough to accommodate even additive space-time white noise, see Remark \ref{rem:wn}. This is a new feature compared to the regularity analysis of \cite{JR,KrLa} for the memoryless case.

The Appendix contains some technical results on the resolvent family $\{S(t)\}_{t\ge 0}$ of the linear, deterministic, homogeneous equation. We would like to highlight the smoothing property \eqref{eq:estSDot} of the derivative $\dot{S}$ which is interesting on its own due to the generality of $b$. The proof relies on Lemma \ref{lem:smu} where certain estimates on the second derivative of the corresponding scalar equation is derived.

\section{Preliminaries}\label{sec:prelim}
In this section we collect some background material, introduce notation and state the main hypothesis on the data in \eqref{eq:problem}.
\subsection{Infinite dimensional stochastic background}
   Let $Q$ be a positive semidefinite bounded linear operator on some separable Hilbert space $U$ with finite trace. Let $\{W(t)\}_{t\ge 0}$ be a $U$-valued Wiener process with covariance operator $Q$ (Q-Wiener process for short) on a probability space $(\Omega, \mathcal{F}, \mathbb{P})$. We equip the probability space with the normal filtration generated by $W$.  Let $Q^{\frac12}$ denote the unique positive semidefinite square root of $Q$. The so-called
\textit{Cameron-Martin} space is defined as $U_0 := Q^{\frac12}(U)$ with inner product
\[
 \inner[U_0]{u_0}{v_0} = \inner[U]{Q^{-\frac12}u_0}{Q^{-\frac12}v_0},\quad u_0,v_0 \in
 U_0,
\]
and induced norm $\|\cdot\|_{U_0}$, where $Q^{-\frac12}$ denotes the pseudo inverse of $Q^{\frac12}$ in case it
is not one-to-one. Let $L_2^0$ denote the space of Hilbert Schmidt operators $T:U_0\to H$ endowed with the norm
$$
\|T\|_{L^0_{2}}^2=\sum_{k=1}^\infty\|Tf_k\|^2
$$
where $\{f_k\}$ is any orthonormal basis for $U_0$.
The next result is a generalized version of It\^o's Isometry which can be found in \cite[Lemma 7.2]{DPZ}. The version cited here is adopted from \cite[Lemma 3.3]{KrLa}.
\begin{proposition}\label{prop:ito}
Let $p\ge 2$ and $\{\Phi(t)\}_{t\in [t_1,t_2]}$ be a $L_2^0$-valued predictable process such that
$$
\left\|\left(\int_{t_1}^{t_2}\|\Phi(s)\|^2_{L_2^0}\,\dd s\right)^{\frac12}\right\|_{L^p(\Omega;\mathbb{R})}<\infty.
$$
Then, there is $C(p)>0$ such that
$$
\left\|\int_{t_1}^{t_2}\Phi(s)\,\dd W(s)\right\|_{L^p(\Omega;H)}\le C(p)\left\|\left(\int_{t_1}^{t_2}\|\Phi(s)\|^2_{L_2^0}\,\dd s\right)^{\frac12}\right\|_{L^p(\Omega;\mathbb{R})}.
$$
\end{proposition}
\begin{remark}\label{rem:UH}
In practice, one usually starts with some covariance operator $\tilde{Q}:H\to H$ not necessarily of trace class to describe the correlation structure of the noise. Here $H$  is the same Hilbert space where the solution $u$ of \eqref{eq:problem} is defined. One then chooses a Hilbert space $U$ and an embedding $J:\tilde{Q}^{\frac12}(H)\to U$ such that $J$ is a Hilbert-Schmidt operator. This is always possible. Then, in a standard way, one can define a Wiener process with values in $U$ with the trace-class covariance operator $Q:=JJ^*$ using an orthogonal series and independent Brownian motions. What is more important though is that the Cameron-Martin space of $JJ^*$ is isometrically isomorphic to the Cameron Martin space of $\tilde{Q}$. Therefore, while the Wiener process $W$ has values in some Hilbert space $U$ which can be chosen many ways, the space of processes that one can integrate with respect to $W$ and the integral itself is independent of the choice of $U$ and the embedding $J$. For a particular instance of this construction, we refer to Remark \ref{rem:wn}.
\end{remark}
Finally we recall a version of Kolmogorov's continuity theorem which is a simplified version of \cite[Theorem 1.4.1]{Kunita}. To measure time regularity of functions defined on an interval $[t_1,t_2]$ with values in a normed space $B$ with norm $\|\cdot\|_{B}$ we introduce the H\"older spaces $C^{\alpha}([t_1,t_2];B)$, $0<\alpha<1$, equipped with the seminorm
$$
\|f\|_{C^{\alpha}([t_1,t_2];B)}=\sup_{s,t\in [t_1,t_2]}\frac{\|f(t)-f(s)\|_{B}}{|t-s|^{\alpha}}.
$$
\begin{proposition}\label{prop:kol}
Let $\{X(t)\}_{t\in [0,T]}$ be a stochastic process with values in a Banach space $B$. If, for some $\alpha,p >0$ with $\alpha p>1$ we have that $X\in C^{\alpha}([0,T]; L^p(\Omega;B))$, then $X$ has a continuous modification $\tilde{X}$ with $\tilde{X}\in L^p(\Omega;C^{\beta}([0,T];B))$ for all $\beta<\alpha-\frac{1}{p}$. Furthermore, if $X(t_0)\in L^p(\Omega;B)$ for some $t_0\in [0,T]$, then $\tilde{X}\in L^p(\Omega;C([0,T];B))$.
\end{proposition}

\subsection{Assumptions}\label{sc:ass}
We first state the assumptions on the convolution kernel $b$. Recall that a kernel $b$ is $k$\textit{-regular} if its Laplace transform $\hat{b}$ satisfies $$|\lambda|^j |\hat b^{(j)}(\lambda)|\le C|\hat b(\lambda)|$$ for all $\mathrm{Re}\lambda>0$ and all $0\le j\le k$ and some $C>0$ (see, \cite[Definition 3.3]{pruss}). It is \textit{sectorial} of angle less than $\pi/2$ if $\sup \{ | \mathrm{arg} \, \hat b(\lambda) |, \; \Re\lambda >0 \}<\pi/2$.
Finally, the kernel $b$ is called $k$\textit{-monotone} ($k\ge 2$) if $b$ is $(k-2)$-time continuously differentiable on $(0,\infty)$, $(-1)^nb^{(n)}(t)\ge 0$ for $t>0$ and $0\le n\le k-2$ and $(-1)^{k-2}b^{(k-2)}$ is nonincreasing and convex, see \cite[Definition 3.4]{pruss}.
\begin{ass}
	\label{ass:convKernel}
The kernel  $0\neq b\in L^1_{\loc}(\R_+)$ is of subexponential growth, 2-regular, sectorial of angle less than $\pi/2$ and the boundary function of its Laplace transform $g(k)=\lim_{\ep\to0}\hat b(\ep+ik)$ satisfies the following growth conditions.
\begin{enumerate}
\item $k\mapsto g(k)/(|k|+|g(k)|)\in L^p(\R)$ for some $1\le p<\infty$.
 \item There exists $C>0$ and $1<\rho<2$ such that for all $\mu>0$
\begin{equation}\label{weakboundhatbAss}
  \int _0^\infty \frac{|g(k)|} {(|k|+\mu|g(k)|^2}\,dk \le C/\mu
\end{equation}
and
\begin{equation}\label{strongboundhatbAss}
   \int _0^\infty \frac{|k^2 g'''(k)|+|k g''(k)|+|g'(k)|+1/\mu} {(|k|+\mu|g(k)|)^2}\,dk \le C/\mu^{1+1/\rho}.
   \end{equation}
\end{enumerate}

\end{ass}
\begin{remark}\label{rem:realb}
We would like to mention that if $b$ is $4$-monotone, $\lim_{t\to \infty}b(t)=0$, and
\begin{equation}\label{eq:b-smooth}
\limsup_{t\rightarrow 0,\infty} \frac{\frac 1t\int_0^t s b(s) \, \dd s}{\int_0^t -s \dot{b}(s) \, \dd s} < +\infty,
\end{equation}
then Assumption \ref{ass:convKernel} is satisfied  with
\begin{equation}\label{eq:sector}
\rho  := 1 + \frac{2}{\pi}\sup \{ | \mathrm{arg} \, \hat b(\lambda) |, \; \Re\lambda >0 \} \in (1,2),
\end{equation}
see Lemma \ref{lem:sd}.
An important example, as mentioned in the introduction is the kernel $b(t)=\frac{1}{\Gamma(\rho-1)}t^{\rho-2}e^{-\eta t}$, $1<\rho<2$ and $\eta\ge 0$. When $\eta=0$, then the corresponding equation \eqref{eq:problem} can be viewed as a fractional-in-time stochastic equation. Note however that, in general, the kernel does not have to be analytic, highly smooth or even positive. As an example, for $1<\rho<2$, consider the kernel with finite history defined by
\begin{equation}\label{eq:bfin}
b(t)=
\begin{cases}
(t^{(\rho-2)/3}-1)^3&\text{ for }0<t<1\\
0&\text{ for }t\geq 1.
\end{cases}
\end{equation}
Then $b$ is 4-monotone, $\lim_{t\to \infty}b(t)=0$, and \eqref{eq:b-smooth} is satisfied. Furthermore $\rho$ specified in \eqref{eq:sector} coincides with $\rho$ given in \eqref{eq:bfin}.

Another example is the function with Laplace transform
$$\hat b(\lambda)=\frac{1}{\lambda^{0.4}+0.4\left(\frac{1}{(\lambda+1)^5}-1\right)}.$$
One can check numerically that the Laplace transform satisfies Assumption \ref{ass:convKernel} but is not completely monotonic as its fourth derivative is not positive. Furthermore, $\rho$ in \eqref{strongboundhatbAss} is equal to $1.4$ which is smaller than the $\rho$ obtained via the sectorial formula \eqref{eq:sector}, around $\approx 1.874$.
\end{remark}

Having symmetric elliptic operator in mind we consider the following assumption on $A$.
\begin{ass}\label{ass:A}
The operator $A:\mathcal{D}(A)\subset H\to H$ is an unbounded, densely defined, linear, self adjoint operator with compact inverse.
\end{ass}
Next, we introduce fractional order spaces and norms.  It is well
known that our assumptions on $A$ imply the existence of a sequence of nondecreasing
positive real numbers $\{\lambda_k\}_{k\geq 1}$ and an orthonormal
basis $\{e_k\}_{k\geq 1}$ of $H$ such that
\begin{equation*}\label{eq:spectral}
Ae_k = \lambda_k e_k, \quad \lim_{k\rightarrow +\infty} \lambda_k = +\infty.
\end{equation*}
In a standard way we introduce the fractional powers $A^s$, $s
\in \mathbb{R}$, of $A$ as
\begin{equation*}\label{eq:fp}
A^s v=\sum_{k=1}^{\infty}\lambda_k^s(v,e_k)e_k,\quad
D(A^s)=\Big\{v\in H:\|A^sv\|^2=\sum_{k=1}^{\infty}\lambda_k^{2s}(v,e_k)^2<\infty\Big\}
\end{equation*}
and spaces $\dot{H}^\beta=D(A^{\beta/2})$ with inner product $(
u , v)_{\dot{H}^\beta}=( A^{\frac{\beta}{2}}u ,
A^{\frac{\beta}{2}}v)$ and induced norms
$\|v\|_{\dot{H}^\beta}=\norm{A^{\beta/2} v}$. More precisely, when $\beta<0$ we set $\dot{H}^\beta$ to be the completion of $H$ with respect to the norm $\|\cdot\|_{\dot{H}^\beta}$. The spaces $\dot{H}^\beta$ are Banach spaces and, for $\beta>0$, the space $\dot{H}^{-\beta}$ is isometrically isomorphic to the dual space of $\dot{H}^{\beta}$.
It is well-known that if $-A$ is the Dirichlet Laplacian on $H=L^2(\mathcal{O})$, where $\mathcal{O}$ is a bounded domain in $\mathbb{R}^d$ with smooth boundary, then for $0\le
\beta < 1/2$ we have that $\dot{H}^\beta=H^\beta$ and for $1/2<\beta\le 2$
 that $\dot{H}^\beta=\{u\in H^\beta:u|_{\partial \mathcal{O}}=0\}$,
where $H^\beta$ denotes the standard Sobolev space of order $\beta$.\\
With the fractional order spaces introduced we let $L^2_{0,r}$ denote the space of Hilbert-Schmidt operators $T:U_0\to \dot{H}^r$ endowed with its natural norm
$$
\|T\|_{L^0_{2,r}}^2=\sum_{k=1}^\infty\|Tf_k\|_{\dot{H}^r}^2,
$$
where $\{f_k\}$ is any orthonormal basis for $U_0$.\\
With the above preparation we make assumptions on the data analogous to those for parabolic equations \cite{JR,KrLa,P2001} taking into account the smoothing effect of the resolvent operator $S$ so that the strict Lipschitz assumptions can be relaxed. The degree of smoothing heavily relies of the convolution kernel $b$ and hence the appearance of the parameter $\rho$ from Assumption \ref{ass:convKernel}.
\begin{ass}\label{ass:id}
Let $\rho\in(1,2)$ as in Assumption \ref{ass:convKernel}, $r<1$ and $p\ge 2$. We suppose that the initial data $u_0$ is $\dot H^{r+\frac1\rho}$-valued $\mathcal{F}_0$-measurable with $u_0\in L^p(\Omega;\dot H^{r+\frac1\rho})$.
\end{ass}
For a discussion on the initial regularity, see Remark \ref{rem:asss}.
\begin{ass}
	\label{ass:main}
	Let $\rho\in(1,2)$ as in Assumption \ref{ass:convKernel} and $r<1$. Suppose that  $F: \dot{H}^{r-1+\frac1\rho}\to \dot{H}^{-1+r}$ and $G:\dot{H}^{r-1+\frac1\rho}\to L_2^0$ satisfy $G(\dot H^{r-1+\frac1\rho})\subset L_{2,r-1+\frac1\rho}^0$ and
	\begin{align}
		\|F(x)-F(y)\|_{\dot{H}^{-1+r}}&\leq C\|x-y\|_{\dot{H}^{r-1+\frac1\rho}},\label{eq:ass1}\\
		\|G(x)-G(y)\|_{L_2^0}&\leq C\|x-y\|_{\dot{H}^{r-1+\frac1\rho}},\label{eq:ass2}\\
		\|G(x)\|_{L_{2,r-1+\frac1\rho}^0}&\leq C(1+\|x\|_{\dot H^{r-1+\frac1\rho}})\label{eq:ass3}
	\end{align}
\end{ass}
Note that \eqref{eq:ass1} implies the linear growth bound
$$
\|F(x)\|_{\dot{H}^{-1+r}}\le C(1+\|x\|_{\dot{H}^{r-1+\frac1\rho}}),\quad x\in \dot{H}^{r-1+\frac1\rho}.
$$
\begin{remark}
If $F:H\to H$ and $G:H\to L_2^0$ are globally Lipschitz continuous, then one may choose
$r=1-\frac{1}{\rho}$ and then \eqref{eq:ass2} implies \eqref{eq:ass3}.
\end{remark}
For a discussion and for specific examples of the type of conditions appearing in  Assumption \ref{ass:main} we refer to \cite{JR}.

\subsection{Resolvent family}
Under the assumptions on $A$ and $b$ it follows from \cite[Corollary 1.2]{pruss} that there exists a strongly continuous family $\{S(t)\}_{t\ge 0}$ such that the function $u(t)=S(t)u_0$, $u_0\in H$, is the unique solution of
$$
u(t)+A\int_0^tB(t-s)u(s)\,\dd s=u_0,\quad t\ge 0,
$$
with $B(t)=\int_0^tb(s)\dd s$. In fact, \cite[Theorem 3.1]{pruss} even shows that $t\to u(t)=S(t)u_0$ is differentiable for $t>0$ and $u_0\in H$ and hence it is the unique solution of
$$
\dot{u}(t)+A\int_0^tb(t-s)u(s)\,\dd s=0,\quad t>0;\quad u(0)=u_0.
$$
Note that the resolvent family does not satisfy the semi-group property because of the presence of a memory term. Nevertheless, it can be written explicitly as
\begin{equation}\label{eq:sk}
S(t) v = \sum_{k=1}^{+\infty} s_k(t) (v,e_k) e_k,\quad, v\in H,
\end{equation}
\noindent where the functions $s_k(t)$ are the solutions of the ordinary differential equations
\begin{equation}\label{eq:skeq}
\dot{s_k}(t) + \lambda_k \int_0^t b(t-s) s_k(s)  \, \dd s = 0, \quad s_k(0)=1,
\end{equation}
with $\{\lambda_k,e_k\}$ being the eigenpairs of $A$.

\section{Main results}\label{sec:mr}
We begin with the definition of mild solution. It is the stochastic version of the deterministic variation of constants formula and is the analogue of the mild solution notion used for SPDEs without a memory term, see \cite[Chapter 7]{DPZ}.
\begin{definition}[Mild solution]
Let $r<1$. We call an $H^{r-1+\frac1\rho}$-valued process $\{u(t)\}_{t\in [0,T]}$ a mild solution of \eqref{eq:problem} if the map $\Phi_{u_0}$ given by
	\begin{align*}
		\Phi_{u_0}(u)(t):=S(t)u_{0}+\int\limits_{0}^tS(t-s)F(u(s))\,\dd
		s+\int\limits_0^tS(t-s)G(u(s))\,\dd W(s),~t\in [0,T],
	\end{align*}
 is well-defined and, for almost all $t\in [0,T]$, we have $\Phi_{u_0}(u)(t)=u(t)$, almost surely.
\end{definition}

The following result, which establishes the smoothing properties of the map $\Phi_{u_0}$, contains the central tools for proving the main result of this paper.
\begin{proposition}
	\label{prop:regularity}
Suppose that Assumptions \ref{ass:convKernel}-- \ref{ass:id} hold together with \eqref{eq:ass1} and \eqref{eq:ass3} in Assumption \ref{ass:main}.
	We set $\kappa=(r-s-1)\rho$.
	Let $T>0$, $p\geq 2$, $u_0\in L^p(\Omega;\dot H^{r+\frac1\rho})$ be $\dot H^{r+\frac1\rho}$-valued $\mathcal{F}_0$-measurable and let $\{u(t)\}_{t\in [0,T]}$ be a predictable process with $$u\in
	L^\infty([0,T];L^p(\Omega,\dot
		H^{r-1+1/\rho})).$$ Then,
	for $s<r-1+\frac2\rho$ and $T>0$, we have that
	\begin{equation}\label{eq:phi1}
		\|\Phi_{u_0}(u)\|_{C^{\min\{\frac12,\frac \kappa2+1\}}([0,T];L^p(\Omega;\dot
		H^{s}))}\leq
		C_T\left(1+\|u\|_{L^\infty([0,T];L^p(\Omega,\dot
		H^{r-1+1/\rho}))}\right).
	\end{equation}
	In particular,
	\begin{equation}\label{eq:phi3}
		\|\Phi_0(u)(t)\|_{L^p(\Omega;\dot H^s)}\leq
		C_Tt^{\min\{\frac12,\frac
		\kappa2+1\}}\left(1+\|u\|_{L^\infty([0,T];L^p(\Omega,\dot
		H^{r-1+1/\rho}))}\right),\quad t\in[0,T].
	\end{equation}
\end{proposition}
\begin{proof}
	We show \eqref{eq:phi1} first. Let $0\le t<t+h\le T$. We first decompose the increments of $\Phi_{u_0}$ as
	\begin{align*}
		\Phi_{u_0}(u)(t+h)-\Phi_{u_0}(u)(t)
		=&\left(S(t+h)-S(t)\right)u_0
		+\int\limits_{t}^{t+h}S(t+h-\sigma)F(u(\sigma))\,\dd \sigma\\
		&+\int\limits_{0}^{t}\left(S(t+h-\sigma)-S(t-\sigma)\right)F(u(\sigma))\,\dd
		\sigma\\
		&+\int\limits_{t}^{t+h}S(t+h-\sigma)G(u(\sigma))\,\dd W(\sigma)\\
		&+\int\limits_0^{t}\left(S(t+h-\sigma)-S(t-\sigma)\right)G(u(\sigma))\,\dd W(\sigma)\\
		=:&I_1+I_2+I_3+I_4+I_5.
	\end{align*}
	To bound $I_1$, we use \eqref{eq:estSDot1} and Assumption \ref{ass:id} to conclude that
\begin{equation}\label{eq:I1}	
\begin{aligned}
		\|I_1\|_{L^p(\Omega;\dot H^s)}
		=&\left\|\int\limits_{t}^{t+h}\dot S(\eta) u_0\,\dd
		\eta\right\|_{L^p(\Omega;\dot H^s)}
		=\left\|\int\limits_{t}^{t+h}A^{\frac {s-r-1/\rho}2}\dot{S}(\eta) A^{\frac{r+1/\rho}2}\,u_0\,\dd
		\eta\right\|_{L^p(\Omega;H)}\\
		\leq&
		C\int\limits_{t}^{t+h}\eta^{\min\{-\frac12,\frac{r+\frac1\rho-s}2-1\}}\,\dd\eta \,\|u_0\|_{L^p(\Omega;\dot
		H^{r+\frac1\rho})}\\
		\leq&
		C\int\limits_{t}^{t+h}\eta^{\min\{-\frac12,\frac{\kappa}2\}}\,\dd\eta \,\|u_0\|_{L^p(\Omega;\dot
		H^{r+\frac1\rho})}\\
&\le Ch^{\min\{\frac12,\frac{\kappa}2+1\}}\|u_0\|_{L^p(\Omega;\dot
		H^{r+\frac1\rho})},
	\end{aligned}
\end{equation}
as $\frac{\kappa}2<\frac{r+\frac1\rho-s}2-1<-\frac12$ if and only if $s>r+\frac1\rho -1$.
	The term $I_2$ can be estimated using \eqref{eq:estS} and \eqref{eq:ass1} from Assumption \ref{ass:main} as
	\begin{align*}
		\left\|I_2\right\|_{L^p(\Omega;\dot H^s)}
		\leq &
		\left\|\int\limits_{t}^{t+h}A^{\frac{s+1-r}2}S(t+h-\sigma)
		A^{\frac{r-1}2}F(u(\sigma))\,\dd
		\sigma\right\|_{L^p(\Omega;H)}\\
		\leq &
		C\int\limits_{t}^{t+h}(t+h-\sigma)^{\min\{0,\frac\kappa2\}}\,\dd \sigma
		\left(1+\|u\|_{L^\infty(0,T;L^p(\Omega;\dot{H}^{r-1+\frac1\rho}))}\right)\\
		\leq & Ch^{
		\min\{1,\frac{\kappa}{2}+1\}}\left(1+\|u\|_{L^\infty(0,T;L^p(\Omega;\dot{H}^{r-1+\frac1\rho}))}\right).
\end{align*}
For $I_3$, by \eqref{eq:estSDot}, \eqref{eq:kap1} and \eqref{eq:ass1} from Assumption \ref{ass:main} we have that
\begin{align*}
		\left\|I_3\right\|_{L^p(\Omega;\dot H^s)}
		=&\left\|\int\limits_0^{t}\int\limits_{t}^{t+h}A^{\frac {s+1-r}2}\dot
		S(\eta-\sigma)\,\dd \eta
		A^{\frac{r-1}2}F(u(\sigma))\,\dd \sigma
		\right\|_{L^p(\Omega;H)}\\
		\leq&
		C\int\limits_{t}^{t+h}\int\limits_0^{t}(\eta-\sigma)^{\min\{-1,\frac{\kappa}{2}-1\}}\,\dd
		\sigma\,\dd\eta
		\left(1+\|u\|_{L^\infty(0,T;L^p(\Omega;\dot{H}^{r-1+\frac1\rho}))}\right)\\
		\leq&
		C\int\limits_{t}^{t+h}\int\limits_0^{t}(\eta-\sigma)^{\min\{-\frac32,\frac{\kappa}{2}-1\}}\,\dd
		\sigma\,\dd\eta
		\left(1+\|u\|_{L^\infty(0,T;L^p(\Omega;\dot{H}^{r-1+\frac1\rho}))}\right)\\
		\leq&
		Ch^{\min\{\frac12,\frac\kappa2+1\}}\left(1+\|u\|_{L^\infty([0,T];L^p(\Omega;\dot{H}^{r-1+\frac1\rho}))}\right).
	\end{align*}
	Similarly, by \eqref{eq:estS}, Proposition \ref{prop:ito} and \eqref{eq:ass3} from Assumption \ref{ass:main}, we obtain
	\begin{align*}
		\|I_4\|_{L^p(\Omega;\dot H^s)}
 		\leq &\left\|\int\limits_{t}^{t+h}A^{\frac{s-r+1-1/\rho}2}S(t+h-\sigma)A^{\frac{r-1+1/\rho}2}G(u(\sigma))\,\dd
 		W(\sigma)\right\|_{L^{p}(\Omega;H)}\\
 		\leq &
 		C\left\|\left(\int\limits_{t}^{t+h}(t+h-\sigma)^{\min\{0,{\kappa+1}\}}\|A^{\frac{r-1+1/\rho}2}
 		G(u(\sigma))\|^2_{L_2^0}\,\dd \sigma
 		\right)^{\frac 12}\right\|_{L^p(\Omega;\mathbb{R})}\\
 		\leq &C
		h^{\min\{\frac12,\frac\kappa2+1\}}\left(1+\|u\|_{L^\infty([0,T];L^p(\Omega;\dot
 		H^{r-1+1/\rho}))}\right).
	\end{align*}
	To bound $I_5$ we use \eqref{eq:estSDot}, \eqref{eq:kap2} together with Proposition \ref{prop:ito} and \eqref{eq:ass3} from Assumption \ref{ass:main} to get
	\begin{align*}
		\|I_5\|_{L^p(\Omega;\dot H^s)}
		= &
		\left\|\int\limits_0^{t}\int\limits_{t}^{t+h}\dot S(\eta-\sigma)\,\dd
		\eta \,G(u(\sigma))\,\dd
		W(\sigma)\right\|_{L^p(\Omega;\dot H^s)}\\
		\leq &
		\int\limits_{t}^{t+h}\left\|\int\limits_0^{t}A^{\frac{s-r+1-1/\rho}2}\dot S(\eta-\sigma)
		A^{\frac{r-1+1/\rho}2}G(u(\sigma))\,\dd
		W(\sigma)\right\|_{L^p(\Omega;\dot H^s)}\,\dd\eta\\
		\leq &
		C\int\limits_{t}^{t+h}\left\|\left(\int\limits_0^{t}(\eta-\sigma)^{\min\{-1,\kappa-1\}}\|
		G(u(\sigma))\|^2_{L_{2,r-1+1/\rho}^0}\,\dd \sigma
		\right)^{\frac 12}\right\|_{L^p(\Omega;\mathbb{R})}\,\dd\eta\\
        \leq &
		C\int\limits_{t}^{t+h}\left\|\left(\int\limits_0^{t}(\eta-\sigma)^{\min\{-2,\kappa-1\}}\|
		G(u(\sigma))\|^2_{L_{2,r-1+1/\rho}^0}\,\dd \sigma
		\right)^{\frac 12}\right\|_{L^p(\Omega;\mathbb{R})}\dd\,\eta\\
		\leq &
		Ch^{\min\{\frac12,\frac\kappa2+1\}}
		\left(1+\|u\|_{L^{\infty}([0,T];L^p(\Omega;\dot
		H^{r-1+1/\rho}))}\right).
	\end{align*}
This finishes the proof of \eqref{eq:phi1}.
Finally \eqref{eq:phi3} follows from \eqref{eq:phi1} noting that $\Phi_0(u)(0)=0$.

\end{proof}
The next theorem constitutes the main result of the paper establishing the existence and uniqueness of smooth mild solutions of \eqref{eq:problem}.
 \begin{theorem}\label{thm:eu}
    Suppose that Assumptions \ref{ass:convKernel}--\ref{ass:main} hold.
	Let $p\geq 2$ and $T>0$.
	Then, there exists a unique mild solution
	\begin{align*}
		u\in C^{\min\{\frac12,\frac\kappa2+1\}}([0,T];L^p(\Omega;\dot H^s))
	\end{align*}
	of \eqref{eq:problem}
	for all $s <r-1+\frac2\rho$ and $\kappa=(r-s-1)\rho$.
\end{theorem}
\begin{proof}
   {\em Step 1: existence and uniqueness.} Let first $s:=r-1+1/\rho$. For $T>0$ we define	
	$$X_T:=\{u\in L^{\infty}([0,T];L^p(\Omega;\dot H^{s})):u\text{ is predictable}\}.$$
Since $\Phi_{u_0}(u)(t)=\Phi_{0}(u)(t)+S(t)u_0$ we have, by \eqref{eq:phi3} of Proposition \ref{prop:regularity}, that $\Phi_{u_0}(u)\in X_T$ if $u\in X_T$.	
Next, for some $\alpha>0$ we introduce an equivalent norm on $X_T$ by
$$
\|u\|_{X_{T,\alpha}}:=\sup_{t\in [0,T]}\left(e^{-\alpha t}\|u(t)\|_{L^p(\Omega:\dot{H}^s)}\right).
$$
We will show that $\Phi_{u_0}$ is a contraction in $X_T$ with respect to the norm $\|\cdot\|_{X_{T,\alpha}}$
	for	suitable $\alpha>0$; that is,  we show that there is $\alpha>0$ such that
	\begin{align*}
	\|\Phi_{u_0}(u)-\Phi_{u_0}(v)\|_{X_{T,\alpha}}\leq K\| u-v\|_{X_{T,\alpha}},\quad K<1.
	\end{align*}
	In order to estimate the norm of the difference,
	we write
	\begin{align*}
		\Phi_{u_0}(u)(t)-\Phi_{u_0}(v)(t)
		=&
		\int\limits_{0}^{t}S(t-\sigma)\left(F(u(\sigma))-F(v(\sigma))\right)\dd \sigma\\
		&+\int\limits_{0}^{t}S(t-\sigma)\left(G(u(\sigma))-G(v(\sigma))\right)\dd
		W(\sigma)\\
		=:&I_1+I_2.
	\end{align*}
	
Let $1<\tilde{q}<2$ fixed and let $\tilde{p}$ such that $\frac{1}{\tilde{p}}+\frac{1}{\tilde{q}}=1$. Then, using
	Assumption~\ref{ass:main} and \eqref{eq:estS} from Lemma~\ref{lem:estSolOp} together with H\"older's inequality we have
	\begin{align*}
		&\|I_1\|_{X_{T,\alpha}}\\
		&\leq\sup\limits_{t\in[0,T]}
		\|\int\limits_{0}^{t}e^{-\alpha(t-\sigma)}e^{-\alpha \sigma}A^{\frac
		{1}{2\rho}}S(t-\sigma)A^{\frac{r-1}2}\left(F(u(\sigma))-F(v(\sigma))\right)\,\dd
		\sigma\|_{L^p(\Omega; H)}\\
		&\leq \sup\limits_{t\in[0,T]}C \int\limits_0^te^{-\alpha(t-\sigma)} (t-\sigma)^{-\frac 12}\,\dd
		\sigma \,\| u-v\|_{X_{T,\alpha}}	\leq C\left(\frac{1}{\alpha\tilde{p}}\right)^{\frac{1}{\tilde p}}T^{-\frac12+\frac{1}{\tilde q}}\| u-v\|_{X_{T,\alpha}}.
	\end{align*}
Similarly, let $\bar{q}> 1$ such that $\bar{q}s\rho<1$. This is possible as, since $r<1$, we have $s\rho=r\rho-\rho+1<1$. Let $\bar{q}$ such that $\frac{1}{\bar{p}}+\frac{1}{\bar{q}}=1$. Then, using again
	Assumption~\ref{ass:main} and \eqref{eq:estS} from Lemma~\ref{lem:estSolOp} together with H\"older's inequality, we have
	\begin{align*}
		&\|I_2\|_{X_{T,\alpha}}\\
		&\leq\sup\limits_{t\in[0,T]}
		\left\|\left(\int\limits_0^{t}e^{-\alpha(t-\sigma)}e^{-\alpha \sigma}\|A^{\frac
		{s}2}S(t-\sigma)
		\left(G(u(\sigma))-G(v(\sigma))\right)\|_{L_2^0}^2\,\dd
		\sigma\right)^{\frac12}\right\|_{L^p(\Omega)}\\
		&\leq C\sup\limits_{t\in[0,T]}
		\left(\int\limits_0^{t}e^{-\alpha(t-\sigma)}(t-\sigma)^{\min\{0,- s\rho\}}
		\,\dd\sigma
		\right)^{\frac12}\|u-v\|_{X_{T,\alpha}}\\
		&\leq C\left(\frac{1}{\alpha\bar{p}}\right)^{\frac{1}{\bar p}}T^{\min\{\frac12,\frac{-\bar{q}s\rho+1}{2}\}}
		\|u-v\|_{X_{T,\alpha}}.
	\end{align*}
Thus, choosing $\alpha>0$ large enough it follows that $\Phi_{u_0}$ is a contraction with respect to the norm $\|\cdot\|_{X_{T,\alpha}}$ and hence, by
Banach's fixed point theorem, there is unique solution $u\in X_T$.\\
{\em Step 2: regularity.} Finally, the smoothing estimate \eqref{eq:phi1} of Proposition~\ref{prop:regularity} yields $$u\in C^{\min\{\frac12,\frac\kappa2+1\}}([0,T];L^p(\Omega;\dot H^s))$$ for all $s<r-1+\frac2\rho$ and $\kappa=(r-s-1)\rho$.
\end{proof}
\begin{remark}
The mild solution in Theorem \ref{thm:eu} is only unique as an element of the space $X_T$.
\end{remark}
\begin{remark}
The regularity obtained in Theorem \ref{thm:eu} is consistent with the results form \cite{JR,KrLa} in the sense that when $b(t)=\frac{1}{\Gamma(\rho-1)}t^{\rho-2}$, then in the limit $\rho\to 1$ we recover the results in memoryless case.
\end{remark}
\begin{corollary}[Pathwise regularity]\label{cor:path}
Suppose that Assumptions \ref{ass:convKernel}--\ref{ass:main} hold.
Let $s<r-1+\frac2\rho$, $\kappa=(r-s-1)\rho$ and $p> 2$ is such that  $\min\{\frac12,\frac\kappa2+1\}-\frac{1}{p}>0$. Then the unique mild solution $u$ of \eqref{eq:problem} has a continuous modification $\tilde{u}$ with  $\tilde{u}\in L^p(\Omega;C^{\beta}([0,T];\dot{H}^s))$ for all $\beta<\min\{\frac12,\frac\kappa2+1\}-\frac{1}{p}$. In particular, $\tilde{u}\in L^p(\Omega;C([0,T];\dot{H}^s))$.
\end{corollary}
\begin{proof}
The statement follows from Theorem \ref{thm:eu} and Proposition \ref{prop:kol}.
\end{proof}
\begin{remark}[The border case $s=r-1+\frac{2}{\rho}$]\label{ss:bcase} Here, we briefly comment on the on the mean-$p$ continuity of $u$ with values in $\dot{H}^{r-1+\frac{2}{\rho}}$. If $r=1-\frac{1}{\rho}$, then we are in the classical global Lipschitz situation for $G$ and the linear growth condition
\eqref{eq:ass3} follows from \eqref{eq:ass2}. Hence both the Lipschitz and linear growth estimate for $G$ is given with respect to the same norms. If Assumption \ref{ass:id} on the initial data holds for some $p\ge 2$, the using the already established fact from Theorem \ref{thm:eu} that $ u\in C^{\frac{1}{2}}([0,T];L^p(\Omega; H))$ together with the smoothing estimates \eqref{eq:estS} and \eqref{eq:estSDot}, and Lemma \ref{lem:cont} one can verify that $u\in C([0,T];L^p(\Omega,\dot{H}^{\frac 1\rho})$ using analogous arguments as in the semigroup case in \cite{KrLa}. For treating the general case $r<1$ one has to impose two extra conditions. Firstly, the extra assumption $G\in C(\dot{H}^{r-1+\frac{1}{\rho}};L^2_{0,r-1+\frac{1}{\rho}})$ (but not necessarily Lipschitz) has to be fulfilled as for general $r$ the linear growth and Lipschitz conditions are not compatible. Secondly, Assumption \ref{ass:id} on the initial data has to hold for some $p> 2$ in order to obtain a version $\tilde{u}$ of $u$ with $\tilde{u}\in L^p(\Omega;C([0,T];\dot{H}^{r-1+\frac{1}{\rho}}))$ via Corollary \ref{cor:path}. Then, again one can argue in a similar fashion as in \cite{KrLa}, using that $ u\in C^{\frac{1}{2}}([0,T];L^p(\Omega,\dot{H}^{r-1+\frac 1\rho}))$ together with the smoothing estimates \eqref{eq:estS} and \eqref{eq:estSDot}, and Lemma \ref{lem:cont}, to conclude that $u\in C([0,T];L^p(\Omega,\dot{H}^{r-1+\frac 2\rho})$.
\end{remark}
\begin{remark}[Initial data]\label{rem:asss}
Suppose that $b$ satisfies the sufficient conditions (for Assumption \ref{ass:convKernel}) laid out in Remark \ref{rem:realb}
and, in addition, the kernel $b$ further obeys
\begin{equation}\label{ass:assymp}
\hat{b}(\lambda)\sim \lambda^{1-\rho}\mbox{ as }\lambda\to\infty.
\end{equation}
An example of a kernel $b$ satisfying this is $b(t)=\frac{1}{\Gamma(\rho-1)}t^{\rho-2}e^{-\eta t}$, $1<\rho<2$ and $\eta\ge 0$.
Let us note that thanks to \cite[estimate (3.6)]{MP97}, we always have that
\begin{equation}\label{eq:omega0}
\hat b(\lambda) \geq c \lambda^{1-\rho},~ \lambda > 1.
\end{equation}
\noindent where $\rho\in(1,2)$ is defined in \eqref{eq:sector}. Therefore, the additional
the assumption \eqref{ass:assymp} on the asymptotic behaviour of $\hat{b}$ really reads as
\begin{equation*}
\hat{b}(\lambda)\le C \lambda^{1-\rho},~\lambda>1.
\end{equation*}
It follows from \eqref{ass:assymp} via a Tauberian theorem for the Laplace transform, that
\begin{equation}\label{eq:l1b}
\|b\|_{L^1(0,t)}\le Ct^{\rho-1}.
\end{equation}
This estimate can be used to prove \eqref{eq:estSDot2} in  Lemma \ref{lem:estSolOp} which then shows an improved smoothing behavior of $\dot{S}$ compared to \eqref{eq:estSDot1}. The only place where \eqref{eq:estSDot1} gets used is the estimate \eqref{eq:I1} where we could then use \eqref{eq:estSDot2} instead. Therefore, if \eqref{ass:assymp} is fulfilled then the assumption $u_0\in L^p(\Omega;\dot{H}^{r+\frac{1}{\rho}})$ on the initial data from Assumption \ref{ass:id} can be replaced by a weaker assumption $u_0\in L^p(\Omega;\dot H^{r-1+\frac2\rho})$ (which in fact corresponds to the regularity of $u$) throughout the paper.
\end{remark}
\subsection{Additive noise}\label{ss:an} Although the conditions on the data are tailored to the semilinear, multiplicative noise case, we state the regularity results for the special case of additive noise with zero initial data and $F\equiv 0$. We set $U=H$ and define
$$
W_S(t):=\int_0^tS(t-s)\,\dd W(s).
$$
These results are interesting on their own and, for example, in \cite{Barbu}, the existence of maximal type spatial regularity results for $W_S$ are part of the set of hypothesis for obtaining existence and uniqueness results for the nonlinear equations. Here we give a set of conditions on $b$ and $Q$ which imply a certain spatial regularity. The result is consistent with the pointwise-in-time spatial regularity result \cite[Theorem 3.6]{kovacsprintems} but, as usual for these kind of results, the right endpoint for the regularity
interval not included.
\begin{corollary}[Additive trace class noise]\label{cor:tcn}
Suppose that Assumption \ref{ass:convKernel} holds and assume further that $Q:H\to H$ is such that $A^{\frac{r-1+\frac{1}{\rho}}{2}}Q^{\frac12}$ is a Hilbert-Schmidt operator on $H$ for some $r\in [1-\frac{1}{\rho},1)$. Let $s<r-1+\frac2\rho$ and $\kappa=(r-s-1)\rho$. Then, there is a continuous modification $\tilde{W}_S$ of $W_S$. Furthermore, if $p> 2$ is such that $\min\{\frac12,\frac\kappa2+1\}-\frac{1}{p}>0$, then $\tilde{W}_S\in L^p(\Omega;C^{\beta}([0,T];\dot{H}^s))$ for all $\beta<\min\{\frac12,\frac\kappa2+1\}-\frac{1}{p}$. In particular, $\tilde{W}_S\in L^p(\Omega;C([0,T];\dot{H}^s))$ for all $p\ge 1$; that is, there is $C=C(T,s,p)>0$ such that
\begin{equation}\label{eq:adn}
\mathbb{E}\sup_{t\in [0,T]}\|\tilde{W}_S(t)\|_{\dot{H}^s}^p\le C,\text{ for all } s<r-1+\frac2\rho.
\end{equation}
\end{corollary}
\begin{proof}
The statement follows from Corollary \ref{cor:path} when setting $U=H$, $u_0=0$, $F\equiv 0$ and $G(x)\equiv I$.
\end{proof}
We end this section with a discussion of a special instance of \eqref{eq:problem} to show the flexibility of the abstract framework.
\begin{remark}[Space-time white noise]\label{rem:wn}
As in Corollary \ref{cor:tcn}, suppose that Assumption \ref{ass:convKernel} holds.
 Let $H=L^2(0,1)$, let $A$ be the Dirichlet Laplacian on $L^2(0,1)$ with $\mathcal{D}(A)=H^2(0,1)\cap H_0^1(0,1)$, and let $u_0=0$, $F\equiv 0$ and $G(x)\equiv I$ (where $I$ is the identity of $H$). We would like to consider a driving process $W$ given by the formal series $W(t)=\sum_{k=1}^{\infty}f_k\beta_k(t)$, where $\beta_{k}$ are real valued independent standard Brownian motions and $\{f_k\}$ is an orthonormal basis of $H$. With the notation of Remark \ref{rem:UH} we have $\tilde{Q}=I$, which is not of trace class. We choose the space $U=\dot{H}^{-\gamma}$ for $\gamma>\frac12$ and the embedding $J:\tilde{Q}^{\frac12}(H)=H\to U$ given by  $J=A^{-\frac{\gamma}{2}}$. Taking into account the asymptotics of the eigenvalues of $A$ we have that $J$ is Hilbert-Schmidt and hence $Q:=JJ^*$ is trace class. Then, as already mentioned in Remark \ref{rem:UH} the Cameron-Martin space of $Q$ and the Cameron-Martin space of $\tilde{Q}$ are isometrically isomorphic and hence $U_0\cong H$. Then assumption \eqref{eq:ass3} is satisfied if
$$
\|I\|_{L_{2,r-1+\frac1\rho}^0}=\sum_{k=1}^{\infty}\|A^{\frac{r-1+\frac1\rho}{2}}e_k\|^2=\sum_{k=1}^{\infty}\lambda_k^{r-1+\frac1\rho}
\sim \sum_{k=1}^{\infty}k^{2(r-1+\frac1\rho)}<\infty,
$$
whence $r<\frac12-\frac{1}{\rho}$. Therefore, using Corollary \ref{cor:path}, we have that there is a continuous modification $\tilde{W}_S$ of $W_S$. Furthermore, if $s<\frac{1}{\rho}-\frac12$, $\kappa=(-\frac12-\frac{1}{\rho}-s)\rho$ and if $p> 2$ is such that $\min\{\frac12,\frac\kappa2+1\}-\frac{1}{p}>0$, then $\tilde{W}_S\in L^p(\Omega;C^{\beta}([0,T];\dot{H}^s))$ for all $\beta<\min\{\frac12,\frac\kappa2+1\}-\frac{1}{p}$. In particular, $\tilde{W}_S\in L^p(\Omega;C([0,T];\dot{H}^s))$ for all $p\ge 1$; that is, there is $C=C(T,s,p)>0$ such that
\begin{equation}\label{eq:adn1}
\mathbb{E}\sup_{t\in [0,T]}\|W_S(t)\|_{\dot{H}^s}^p\le C\text{ for all }s<\frac{1}{\rho}-\frac12.
\end{equation}
\end{remark}
\begin{remark}
The maximal type estimates \eqref{eq:adn} and \eqref{eq:adn1} are consistent with the pointwise-in-time bounds obtained in \cite[Theorem 3.6]{kovacsprintems} for $p=2$.
\end{remark}
\section{Appendix}
Here we first derive the crucial bounds on the second derivative of the solution of the scalar problem
\begin{align}
	\label{eq:functions}
	\dot s_\mu(t)+\mu \int\limits_{0}^t b(t-s)s_\mu\dd s=0,\quad s_\mu(0)=1,\quad \mu>0.
\end{align}
These then yield the norm bound \eqref{eq:estSDot} in Lemma \ref{lem:estSolOp}, on the derivative of the solution of the linear, homogeneous deterministic problem via a spectral decomposition \eqref{eq:sk}. Recall the following important result regarding $L^1$ bounds for Laplace transforms which is essentially follows from \cite[Theorem 4.3]{Hille1935}.
\begin{proposition}\label{Hardy}
  Let $r$ be an analytic function in the right halfplane with boundary function $g(x)=\lim_{\ep\to 0}r(\ep+ix)$ for all $x\in\R$. If $g$ is of bounded variation and $g\in L^p(\R)$ for some $1\le p<\infty$, then
there exists $f\in L^1(\R_+)$ with $\int_0^\infty e^{-\lambda t}f(t)\,dt =r(\lambda)$ and
$$\|f\|_ {L^1(\R_+)}\le  \frac12\|g'\|_ {L^1(\R)}.$$
\end{proposition}

The next lemma shows that it is possible to give a simple set of conditions on the convolution kernel $b$, rather than on $\widehat{b}$, so that the kernel satisfies Assumption \ref{ass:convKernel}.
\begin{lemma}\label{lem:sd}
	If $0\neq b\in L^1_{\loc}(\R_+)$ is $4$-monotone, $\lim_{t\to \infty}b(t)=0$ and satisfies \eqref{eq:b-smooth}, then $b$ satisfies Assumption \ref{ass:convKernel}.
\end{lemma}
\begin{proof}  It follows
 from \cite[Proposition 3.10]{pruss} that for 3-monotone and locally integrable kernels $b$, condition \eqref{eq:b-smooth} is equivalent to \eqref{eq:sector}
and thus $b$ is sectorial of angle less that $\pi/2$. Furthermore, 4-monotonicity implies 3-regularity (\cite[Proposition 3.3]{pruss}).  By \cite[Proposition 3.8]{pruss},
\begin{align*}
     |\hat b^{(n)}(\lambda)|\leq C \int_0^{1/|\lambda|} \tau^n b(\tau)\,d\tau,\quad \mathrm{Re}\lambda>0,\ n=0,1,2,3.
 \end{align*}
 In particular, it implies that $ g$ is bounded in a neighbourhood of infinity and hence $k\mapsto \frac{g(k)}{|k|+|g(k)|}\in L^p(\R)$ for all $1< p<\infty$. Furthermore, it also implies that \begin{align*}
     | \lambda^{n-1}\hat b^{(n)}(\lambda)|\leq C|\lambda|^{n-1} \int_0^{1/|\lambda|} \tau^n b(\tau)\,d\tau\le C\Psi(1/|\lambda|),\quad \mathrm{Re}\lambda>0,\ n=1,2,3,
 \end{align*} where $\Psi(s)=\int_0^st b(t)\,\dd t$, $s>0$. Then identically to the part of the proof of  \cite[Proposition 6]{MP97} that estimates the Hardy norms of $h'_\mu$ and $h''_\mu$ we have  \eqref{weakboundhatbAss} and \eqref{strongboundhatbAss}.
 \end{proof}

\begin{lemma}\label{lem:smu}
Suppose that the convolution kernel $b$ satisfies Assumption \ref{ass:convKernel}.
  Then there
	exists $C_0>0$ such that
	\begin{align}
         \|s_{\mu}\|_{L^{\infty}(\R_+)}&\le 1,&\mu>0;\label{eq:contr}\\
         \|s_{\mu}\|_{L^1(\R_+)}&\le C_0 \mu^{-1/\rho};&
		\mu>0,\label{eq:fl1mu}\\
		\label{eq:functionsEst1}
		\|\dot s_\mu\|_{L^1(\R_+)}+\| t\ddot s_\mu\|_{L^1(\R_+)}&\leq C_0;&
		\mu>0,\\
		\label{eq:functionsEst2}
		\|t\dot s_\mu\|_{L^1(\R_+)}+\| t^2\ddot s_\mu\|_{L^1(\R_+)}&\leq C_0 \mu^{-1/\rho};&
		\mu>0.
	\end{align}
\end{lemma}
\begin{proof}
Since $\Re \hat{b}(\lambda)\ge 0$ for $\Re\lambda>0$, we have that that the kernel $b$ is of positive type and hence \eqref{eq:contr} follows from \cite[Corollary 1.2]{pruss}. The estimate in \eqref{eq:fl1mu} is a direct consequence of the estimates on $\dot{s}_{\mu}$ in \eqref{eq:functionsEst1} and \eqref{eq:functionsEst2} as shown in the proof of \cite[Lemma 3.1]{CDaPP}.
 Thus, we need to show that the Laplace transform  of each of the  left hand terms in \eqref{eq:functionsEst1} and \eqref{eq:functionsEst2} has appropriate bounds; that is, that the boundary function of their Laplace transforms are in $L^p$ for some $1\le p<\infty$ and that its derivative is in $L^1$. Let $F^1_\mu=\widehat{\dot s_\mu}$, $f^1_\mu=\widehat {t\ddot s_\mu}$, $F^2_\mu=\widehat{t\dot s_\mu}$ and $f^2_\mu=\widehat { t^2\ddot s_\mu}$. Then
 $F^1_\mu(\lambda)=\frac{\lambda}{\lambda+\mu\hat b(\lambda)}-1=-\mu\frac{\hat b(\lambda)}{\lambda+\mu\hat b(\lambda)}$ and
 $$f^1_\mu(\lambda)=\frac{d}{d\lambda} \left( \mu\frac{\lambda\hat b(\lambda)}{\lambda+\mu\hat b(\lambda)} +\dot s_\mu(0)\right)=\mu\frac{\mu\hat b(\lambda)^2+\lambda^2 \hat b'(\lambda))}{(\lambda+\mu\hat b(\lambda))^2}, $$
whereas
$$F^2_\mu=-\frac{d}{d\lambda}F^1_\mu=\mu\frac{\lambda\hat{b}'(\lambda)-\hat b(\lambda)}{(\lambda+\mu\hat b(\lambda))^2}$$
and
\begin{equation}\begin{split}
f^2_\mu=&-\frac{d}{d\lambda}f^1_\mu\\ =& \mu \frac{\mu\lambda^2 \hat b''(\lambda)\hat b(\lambda) 		-2\mu\lambda^2( \hat b'(\lambda))^2+\lambda^3  \hat b''(\lambda) 		-4\lambda^2 		\hat b'(\lambda)+2\mu \hat b(\lambda)^2+4\lambda 		\hat b(\lambda)}{(\lambda+\mu \hat b(\lambda))^3}.
\end{split}\end{equation}
We also need estimates for
\begin{equation}\begin{split}
  \frac{d}{d\lambda}F^2_\mu(\lambda)=&\mu\frac{ \lambda \hat b''(\lambda ) (\lambda+\mu \hat b(\lambda ) )+2 \left(\hat b(\lambda )-\lambda  \hat b'(\lambda )\right) \left(\mu  \hat b'(\lambda )+1\right)}{(\lambda +\mu  \hat b(\lambda ))^3}
  \end{split}
\end{equation}
and
\begin{equation}\begin{split}
  \frac{d}{d\lambda}f^2_\mu(\lambda)=&\mu  \lambda ^2\frac{ \lambda ^2 \hat b^{(3)}(\lambda )+6 \mu ^2 \hat b'(\lambda )^3+6 \mu  \hat b'(\lambda )^2-6 \lambda  \mu  \hat b'(\lambda ) \hat b''(\lambda )}{(\lambda +\mu  \hat b(\lambda ) )^4} \\
  &+\mu^2  \hat b(\lambda )^2\frac{ \lambda ^2 \mu  \hat b^{(3)}(\lambda )+6 \lambda  \mu  \hat b''(\lambda )+6 \mu  \hat b'(\lambda )+6}{(\lambda +\mu  \hat b(\lambda ))^4}\\
  &+2 \lambda  \mu^2  \hat b(\lambda )\frac{ -6 \mu  \hat b'(\lambda )^2+\lambda  \left(\lambda  \hat b^{(3)}(\lambda )+3 \hat b''(\lambda )\right)-3 \hat b'(\lambda ) \left(\lambda  \mu  \hat b''(\lambda )+2\right)}{(\lambda +\mu  \hat b(\lambda ))^4}.
  \end{split}
\end{equation}
As $b$ is sectorial, there exists $C$ such that $|\lambda+\mu \hat b(\lambda)|\ge C(|\lambda|+\mu|\hat b(\lambda)|)$. Use the 2-regularity of $b$; that is, $|\lambda^n\hat b^{(n)}(\lambda)|\le C|\hat b(\lambda)|$, $n\le 2$, to obtain that  $F^1_\mu,f^1_\mu\in H^p(\R)$ for some $1\le p<\infty$ and
$$\int_0^\infty\left|F^2_\mu(ik)\right|\,dk\le C\mu \int_0^\infty \frac{|g(k)|}{(|k|+\mu|g(k)|)^2}d\lambda\le C.$$
The same estimate holds for  $f_\mu^2=-\frac{d}{d\lambda} f_\mu^1$ and hence by Assumption \ref{ass:convKernel} and Proposition \ref{Hardy} we have on one hand \eqref{eq:functionsEst1} and on the other, $f_\mu^2$ and $F^2_\mu\in H^1$. Again using 2-regularity it is straight-forward to see that
$$\left|\frac{d}{d\lambda}F^2_\mu(\lambda=ik)\right|+\left|\frac{d}{d\lambda}f^2_\mu(\lambda=ik)\right|\le C\mu \frac{|k|^2|g'''(k)|+ |k||g''(k)|+|g'(k)|+1/\mu}{(|k|+\mu|g(k)|)^2}$$
and hence by Assumption \ref{ass:convKernel} and Proposition \ref{Hardy} we obtain \eqref{eq:functionsEst2}.

\end{proof}
The next set of results specifies the smoothing properties of the solution of the linear, homogeneous, deterministic problem.
\begin{lemma}
	\label{lem:estSolOp}
	Suppose that the convolution kernel $b$ satisfies Assumption \ref{ass:convKernel}. Then
			\begin{align}
				\|A^sS(t)\|_{\cL(H)}&\leq C t^{-s\rho},\quad t>0,
				s\in[0,1/\rho],
				\label{eq:estS}\\
				\|A^s\dot S(t)\|_{\cL(H)}&\leq Ct^{-s\rho-1},\quad
				t>0,~s\in [0,1/\rho],\label{eq:estSDot}\\
                 \|A^{-s}\dot{S}(t)\|&\leq C\|b\|_{L^1(0,t)}^st^{s-1},~s\in [0,1].\label{eq:estSDot1}
\end{align}
In particular, if $b$ satisfies the sufficient conditions (for Assumption \ref{ass:convKernel}) laid out in Remark \ref{rem:realb}  together with \eqref{ass:assymp}, then
\begin{align}				
\|A^{-s}\dot{S}(t)\|_{\cL(H)}&\leq
				C t^{\rho s-1},\quad
				t>0,\ s\in[0,1].\label{eq:estSDot2}
			\end{align}
		\end{lemma}
\begin{proof}
Estimate \eqref{eq:estS} follow from the scalar estimates in Lemma \ref{lem:smu} and we refer to the proof of  \cite[Proposition 2.5]{kovacsprintems} for details. Theorem \cite[Theorem 3.1]{pruss} shows that under Assumptions \ref{ass:A} and \ref{ass:convKernel} there is $C>0$ such that $\|\dot{S}(t)\|\le Ct^{-1}$, for all $t>0$.
Then \eqref{eq:estSDot1} follows easily from \eqref{eq:contr} and again we refer to the proof of  \cite[Proposition 2.5]{kovacsprintems} for further details.
The bound in \eqref{eq:estSDot2} follows from \eqref{eq:estSDot1} using \eqref{eq:l1b} from Remark \ref{rem:asss}.
	To show \eqref{eq:estSDot} first note that $s=0$ is included in \eqref{eq:estSDot2}. Next, we use Lemma \ref{lem:sd} and H\"older's inequality to conclude that, for $0<\delta<1$,
\begin{align*}
\int_{0}^{\infty}t^{\delta+1}|\ddot{s}_\mu(t)|\,\dd t&=\int_0^\infty |t^2\ddot{s}_\mu(t)|^{\delta}|t\ddot{s}_{\mu}(t)|^{1-\delta}\,\dd t\\
&\le \left(\int_0^\infty |t^2\ddot{s}_\mu(t)|\,\dd t\right)^{\delta}\left(\int_0^\infty |t\ddot{s}_\mu(t)|\,\dd t\right)^{1-\delta}\le C_0\mu^{-\frac{\delta}{\rho}},~\mu>0.
\end{align*}
Note that this estimate also holds for $\delta=1$, by \eqref{eq:functionsEst2}.
Therefore, taking into account that $\lim_{t\to \infty}\dot{s}_{\mu}(t)=0$ by \eqref{eq:estSDot1} with $s=0$, we have that
\begin{equation*}\label{eq:dse}
|\dot{s}_{\mu}(t)|=\left|\int_{t}^{\infty}\ddot{s}_\mu(r)\,\dd r\right|=\left|\int_{t}^{\infty}r^{-\delta-1}r^{\delta+1}\ddot{s}_\mu(r)\,\dd r\right|
\le C_0t^{-\delta-1}\mu^{-\frac{\delta}{\rho}}.
\end{equation*}
That is,
$$
|\mu^s\dot{s}_{\mu}(t)|\le C_0t^{-1-\rho s}, ~0< s \le\frac{1}{\rho}.
$$
Hence,
$$
\|A^s\dot{S}(t)x\|^2=\sum_{k=1}^{\infty}\left(\lambda_k^{s}\dot{s}_{\lambda_k}(t)\right)^2(x,e_k)^2\le C_0t^{-2-2\rho s}\|x\|^2
$$
and \eqref{eq:estSDot} follows.

\end{proof}
The next lemma is key to be able to consider the border case in Remark \ref{ss:bcase}.
\begin{lemma}\label{lem:cont}
If the kernel $b$ satisfies Assumption \ref{ass:convKernel} and
			$x\in L^p(\Omega;\dot H^{r-1})$, then
the function $$t\mapsto \int_0^t
			S(\sigma)
			x\,\dd \sigma\in C([0,\infty);L^p(\Omega;\dot H^{r-1+2/\rho})).$$ Moreover, the following estimate holds
			\begin{align}
				\label{eq:estIntS}
				\left\|\int\limits_0^t S(\sigma) x\,\dd \sigma\right\|_{L^p(\Omega;\dot
				H^{r-1+2/\rho})}&\leq C\|x\|_{L^p(\Omega;\dot H^{r-1})},\quad
				t>0.
			\end{align}		
Furthermore, if $G\in L^p(\Omega;L^2_{0,\dot H^{r-1+1/\rho}})$, then  the function $$t\mapsto \int_0^t
			S(t-\sigma)
			G\,\dd W(\sigma) \in C([0,\infty);L^p(\Omega;\dot H^{r-1+2/\rho}).$$
			 Moreover, the following estimate holds
			\begin{align}
				\label{eq:estIntS2}
				\left\|\int\limits_0^t S(t-\sigma) G\,\dd
				W(\sigma)\right\|_{L^p(\Omega;\dot H^{r-1+2/\rho})}&\leq C\|G\|_{L^p(\Omega;L^2_{0,\dot H^{r-1+1/\rho}})},\quad
				t>0.
			\end{align}
\end{lemma}
\begin{proof}
 Estimate \eqref{eq:estIntS} follows from the calculation, taking into account \eqref{eq:fl1mu},
			\begin{align*}
				\left\|\int\limits_0^t S(\sigma) x\,\dd\sigma\right\|^2_{L^p(\Omega;\dot
				H^{r-1+2/\rho})}
				=&\left\|\int\limits_0^t A^{\frac1\rho}S(\sigma)
				A^{\frac{r-1}2}x\,\dd \sigma\right\|^2_{L^p(\Omega;H)}\\
				=&\left\|\left(\sum\limits_{k=1}^\infty
				\int\limits_{0}^t\lambda_k^{\frac1\rho} s_{\lambda_k}(\sigma)\,\dd
				\sigma\right)^2(A^{\frac{r-1}2}x,e_k)^2\right\|_{L^p(\Omega)}\\
				\leq& C \|x\|^2_{L^p(\Omega;\dot
				H^{r-1})}.
			\end{align*}
			Continuity is standard now. Finally, \eqref{eq:estIntS2} can be shown in a similar fashion as \cite[Theorem 3.6]{kovacsprintems} and continuity follows again using standard arguments.
\end{proof}
Finally, we often make use of the following elementary estimate.
\begin{lemma}\label{lem:kappa}
	If $\kappa\in(-2,0)$ and $h,t>0$, then there is $C>0$ such that
	\begin{equation}\label{eq:kap1}	
		\int\limits_{t}^{t+h}\int\limits_0^{t}(\eta-\sigma)^{\frac{\kappa}2-1}\dd
		\sigma\dd\eta
		\leq C h^{\frac\kappa2+1}
\end{equation}
and
	\begin{equation}\label{eq:kap2}
		\int\limits_{t}^{t+h}\left(\int\limits_0^{t}(\eta-\sigma)^{\kappa-1}
		\dd \sigma
		\right)^{\frac 12}\dd \eta
		\leq Ch^{\frac{\kappa}2+1}.
	\end{equation}
\end{lemma}
\begin{proof}
We have that
	\begin{align*}	
		\int\limits_{t}^{t+h}\int\limits_0^{t}(\eta-\sigma)^{\frac{\kappa}2-1}\dd
		\sigma\dd\eta
		&= C
			\int\limits_{t}^{t+h}(\eta-t)^{\frac\kappa2}-\eta^{\frac{\kappa}{2}}
			\dd\eta\leq C\int\limits_{t}^{t+h}(\eta-t)^{\frac\kappa2}\dd \eta=C h^{\frac\kappa2+1}
\end{align*}
and that
\begin{align*}
		\int\limits_{t}^{t+h}\left(\int\limits_0^{t}(\eta-\sigma)^{\kappa-1}
		\dd \sigma
		\right)^{\frac 12}\dd \eta
		=&C
		\int\limits_{t}^{t+h}\left((\eta-t)^{\kappa}-\eta^\kappa
		\right)^{\frac 12}\dd \eta
		\leq C
		\int\limits_{t}^{t+h}(\eta-t)^{\frac{\kappa}{2}}
		\dd \eta
		\leq Ch^{\frac{\kappa}2+1}.
	\end{align*}
\end{proof}

\bibliographystyle{plain}


\end{document}